\documentclass[12pt]{article}
\usepackage{amssymb}
\usepackage{amsmath}
\usepackage{amsthm}
\usepackage{calc}
\usepackage{stackrel}
\usepackage{color}
\usepackage{cancel}
\usepackage{lineno}
\usepackage{hyperref}
\usepackage[affil-it]{authblk}
 
\usepackage{graphicx}

\usepackage{comment} 

\newtheorem{theorem}{Theorem}
\newtheorem{prop}[theorem]{Proposition}
\newtheorem{lem}[theorem]{Lemma}

\newtheorem{defi}[theorem]{Definition}
\newtheorem{remark}[theorem]{Remark}

\newtheorem{example}[theorem]{Example}

\usepackage{cancel}

\begin{document}
\title{On three-dimensional Poisson quasi-Nijenhuis\\ manifolds and Haantjes structures}
\date{} 

\author{E.\ Chu\~no Vizarreta${}^{1}$, I.\ Mencattini${}^2$,  
M.\ Pedroni${}^{3,4}$}

\affil{
{\small  $^1$Unidade Acad\^emica de Belo Jardim, Universidade Federal Rural de Pernambuco, Brazil}\\
{\small eber.vizarreta@ufrpe.br  
}\\
\medskip
{\small $^2$Instituto de Ci\^encias Matem\'aticas e de Computa\c c\~ao,  Universidade de S\~ao Paulo, Brazil}\\
{\small igorre@icmc.usp.br 
}\\
\medskip
{\small $^3$Dipartimento di Ingegneria Gestionale, dell'Informazione e della Produzione,  Universit\`a di Bergamo, Italy}\\
{\small marco.pedroni@unibg.it 
}\\
\medskip
{\small  $^4$INFN, Sezione di Milano-Bicocca, Piazza della Scienza 3, 20126 Milano, Italy}
}

\maketitle
\abstract{\noindent
In this note we first characterize Poisson quasi-Nijenhuis structures on three-dimensional oriented manifolds 
whose underlying Poisson tensor never vanishes. We then apply this result to show that each of these structures is  (locally) a deformation of a PN structure and is involutive. Finally, we prove that every such three-dimensional 
Poisson quasi-Nijenhuis manifold is a Haantjes manifold and that it carries a generalized Lenard-Magri chain.

}  

\baselineskip=0,6cm

\section{Introduction}

The notion of Poisson-Nijenhuis structure, i.e., a \emph{compatible} pair of a Poisson bivector $\pi$ and a 
torsionless tensor field $N$ of type $(1,1)$, is instrumental to the theory of classical integrable systems, since it provides a geometrical framework to produce involutive  families of conserved quantities for the underlying Hamiltonian dynamical system, see \cite{KM, Magri-Morosi}. It is important to stress that this property hinges crucially on the hypothesis that the Nijenhuis torsion of $N$ is equal to zero, see \eqref{eq:torsionA} and Section \ref{sec:2} for the relevant definitions. In fact, even a controlled change on this assumption destroys, in general, the possibility to produce the desired families of first integrals, spoiling a possible use of the newborn geometrical structure in the theory of integrable models.

An example of this state of affairs is provided by the so called Poisson quasi-Nijenhuis structures, i.e., a \emph{compatible} triple of a Poisson bivector $\pi$, a 3-form $\phi$ and a $(1,1)$ tensor field $N$, whose Nijenhuis torsion is given in  \eqref{eq:PqNtor}. 
Although this class of structures represents a very interesting generalization of the Poisson-Nijenhuis one (see \cite{SX, Antunes2008, C-NdC-2010, BursztynDrummondNetto2021, DMP2024}) and has interesting applications in mathematical physics 
(see, for example, \cite{Zucchini}), it did not play any role in the theory of integrable systems until the discover 
of its relation with 
$n$-particle Toda systems, see \cite{FMOP2020,FMP2024}. In these papers the notion of \emph{involutive} Poisson quasi-Nijenhuis structure was introduced, see Definition \ref{def:invpqn} hereafter, and a set of \emph{sufficient} conditions for a Poisson quasi-Nijenhuis structure to be involutive were presented, see Theorem \ref{theo:inv} below. It is worth mentioning that, since these conditions are very stringent, the task of finding (interesting) examples of involutive Poisson quasi-Nijenhuis structures is a non-trivial open problem, see for example \cite{FMP2023}.

In line with the title of \cite{Kos-Recursion}, the attempt to go beyond recursion operators to find a more general theoretical framework than the Poisson-Nijenhuis one 
was recently undertaken in \cite{MagriHaa, MagriVeselov, MagriSymm} and \cite{tempestatondoben, tempestatondo, tempestatondohigher,tempestatondoclass}. In these works, one of the main novelties is the presence of (families of) $(1,1)$ tensor fields whose \emph{Haantjes} torsion is equal to zero. 
As the Nijenhuis one, the Haantjes torsion of  $N$ is a 2-form with values in the tangent bundle of the underlying manifold, see  \eqref{eq:haant}, which controls the Frobenius integrability of the 2-planes generated by pairs of eigenvectors of $N$. 
Its relevance in the theory of integrable systems of hydrodynamic type was already observed in 
\cite{Bogo96-182,FeraMar}, but its importance in the theory of finite-dimensional integrable system was unnoticed until the above mentioned works of Magri, Tempesta and Tondo.

The present note is a \emph{first attempt} to construct a bridge between Poisson quasi-Nijenhuis and Haantjes structures, aiming to find, in a near future, more applications of the former to the theory of finite-dimensional integrable systems.
Our results concern oriented three-dimensional manifolds and generalize the one proved in \cite{Cheng-Sheng} for Poisson-Nijenhuis (PN from now on) manifolds. Assuming that the Poisson tensor never vanishes, we describe any Poisson quasi-Nijenhuis manifold in terms of few geometrical objects. Moreover, we show that such manifolds are all involutive. Finally, we show that they are Haantjes manifolds with a generalized Lenard-Magri chain.

A more detailed plan of the paper is as follows. In Section \ref{sec:2} we first collect some general information about Poisson quasi-Nijenhuis (PqN from now on) structures and we remind the definition of involutive PqN manifold. Then we recall two results, the first giving a set of sufficient conditions 
for the involutivity of a PqN  structure, see Theorem \ref{theo:inv}, the second enclosing a general statement about 
deformations of PqN structures, see Theorem \ref{thm:defthe}. 
Section \ref{sec:3} is devoted to characterize 
PqN structures (with never vanishing Poisson tensor) on an oriented three-dimensional manifold, see 
Theorem \ref{prop 3D PqN}. 
As an application of this result we prove that such structures are (locally) obtained by deforming a PN structure via a closed 2-form, see Proposition \ref{prop:local}, and that they are involutive, see Proposition \ref{thm:inv3}. Moreover, in Section \ref{sec:4} they are shown to be Haantjes structures, see Theorem \ref{theo:haan}, and that they carry a generalized Lenard-Magri chain, see Theorem \ref{theo:Len-Ma-Ge}. In Subsection \ref{ss:tempton} we make a few final comments aiming to link the structures introduced in the recent works \cite{tempestatondoben, tempestatondo, tempestatondohigher,tempestatondoclass,TondoLag} with the geometry of the PqN 
manifolds studied in this note. 
In Section \ref{sec:Appendix} we collect the proof of an identity used in the main text, an alternative proof of Proposition \ref{thm:inv3} and few facts about the Haantjes torsion. We close this manuscript with Section \ref{sec:comfin}, which encloses a few final remarks about the results obtained in this paper and some possible future lines of investigation opened in this work.

\subsection{Notations and conventions} Hereafter we will systematically use the following conventions. All manifolds, 
vector bundles and their sections will be smooth.  The space of sections of the vector bundle $V$ will be denoted with $\Gamma(V)$, with the exception of the sections of $TM$, i.e., vector fields, and of $\Lambda^k(T^\ast M)$, i.e., $k$-differential forms, which will be denoted by $\mathfrak X(M)$ and, respectively, $\Omega^k(M)$. If $V$ is a vector bundle over a manifold $M$ and $A\in\text{End}_{C^\infty(M)}(V)$, $A^\ast$ will denote the endomorphism of $V^\ast$, the dual bundle of $V$, defined by $\langle A^\ast\eta,s\rangle=\langle\eta, A s\rangle$ for all $s\in\Gamma(V)$ and $\eta\in\Gamma(V^\ast)$. 
Every bivector field $P\in\Gamma(\Lambda^2(TM))$ defines a unique (and it is defined uniquely by) $P^\sharp\in\text{Hom}_{C^\infty(M)}(T^\ast M,TM)$ via the  
formula $\langle\beta,P^\sharp\alpha\rangle=P(\alpha,\beta)$, for all $\alpha,\beta\in\Omega^1(M)$. Similarly, every 2-form $\omega\in\Omega^2(M)$ defines a unique (and it is defined uniquely by) $\omega^\flat\in \text{Hom}_{C^\infty(M)}(TM,T^\ast M)$ via $\langle\omega^\flat X,Y\rangle=\omega(X,Y)$, for all $X,Y\in\mathfrak X(M)$. Finally, recall that to every $N\in\text{End}_{C^\infty(M)}TM$, i.e., a tensor field of type $(1,1)$, one can associate its Nijenhuis torsion $T_N$, which is the 2-form on $M$ with values in $TM$ defined by 
\begin{equation}
T_N(X,Y)=[NX,NY]-N([NX,Y]+[X,NY]-N[X,Y]).\label{eq:torsionA}
\end{equation}
The (1,1) tensor field $N$ will be called Nijenhuis torsion-free, or, more simply, torsion-free (or a Nijenhuis operator) 
if $T_N\equiv 0$.

\par\bigskip\noindent
{\bf Acknowledgments.} 
All authors gratefully acknowledge the auspices of the GNFM Section of INdAM and of the Brazilian Centre of Geometry grant No. 2024/00923-6, Fapesp.

\section{Recollection of Poisson quasi-Nijenhuis structures}\label{sec:2}

A Poisson-Nijenhuis structure on a manifold $M$ is a pair $(\pi,N)$ of a Poisson bivector  $\pi$ and a torsion-free tensor field $N$ of type $(1,1)$, satisfying the following compatibility conditions: 
\begin{equation}
\label{eq:pncon}
\begin{aligned}
&\mbox{(C1)}\quad N\circ\pi^\sharp=\pi^\sharp\circ N^\ast=:\pi^\sharp_N;
\\
&\mbox{(C2)}\quad[\alpha,\beta]_{\pi_N}=[N^\ast\alpha,\beta]_{\pi}+[\alpha,N^\ast\beta]_{\pi}-N^\ast[\alpha,\beta]_{\pi},\qquad\forall\alpha,\beta\in\Omega^1(M),
\end{aligned}
\end{equation}
where, if $P$ is a bivector field, 
the bracket $[-,-]_P:\Omega^1(M)\times\Omega^1(M)\rightarrow\Omega^1(M)$ is defined by 
\[
[\alpha,\beta]_P=\mathcal L_{P^\sharp\alpha}\beta
-\mathcal L_{P^\sharp\beta}\alpha-d\langle\beta,P^\sharp\alpha\rangle,
\]
and $\mathcal L_X$ is the Lie derivative with respect to the vector field $X$. A manifold $M$ endowed 
with a PN structure $(\pi,N)$ will be called a PN manifold and will be denoted by the triple $(M,\pi,N)$. 

\begin{remark}\label{rem:MMcon}
The second formula in 
\eqref{eq:pncon} is often written as $C(\pi,N)\equiv 0$, where 
\begin{equation}
C(\pi,N)(\alpha,\beta)=[\alpha,\beta]_{\pi_N}-[N^\ast\alpha,\beta]_{\pi}-[\alpha,N^\ast\beta]_{\pi}+N^\ast[\alpha,\beta]_{\pi},\qquad 
\alpha,\beta\in\Omega^1(M),\label{eq:MMconc}
\end{equation}
is the so called Magri-Morosi concomitant. It is also worth recalling that a torsion-free $(1,1)$ tensor field $N$ is often called a recursion operator since the vector space of all $X\in\mathfrak X(M)$ such that $\mathcal L_X N=0$ is an $N$-invariant Lie subalgebra of $\mathfrak X(M)$. This means, in particular, that if $X$ is an infinitesimal symmetry of $N$, i.e., 
$\mathcal L_XN=0$, then $N^kX$ will also be, for all $k\geq 1$.
\end{remark}
To every PN structure one can associate the space
\begin{equation}
\Omega^1_{Ham}(M)=\{\alpha\in\Omega^1(M)\mid d\alpha=0=d_N\alpha\},
\end{equation}
where $d_N:\Omega^\bullet(M)\rightarrow\Omega^{\bullet+1}(M)$ is defined by $d_N=i_N\circ d-d\circ i_N$ and $i_N$ is the degree-zero endormorphism of $\Omega^\bullet(M)$ defined on $\omega\in\Omega^k(M)$ by $i_N\omega(X_1,\dots,X_k)=\sum_{j=1}^k\omega(X_1,\dots,NX_j,\dots,X_k)$. The following proposition encloses an important property of $\Omega^1_{Ham}(M)$. 
\begin{prop}[\cite{Magri-Morosi,Bonechi}] The space $\Omega^1_{Ham}(M)$ is $N^\ast$-invariant. Moreover,
\begin{enumerate}
\item For every $k\geq 0$, it is closed with respect to $[-,-]_{\pi_{N^k}}$.
\item Given $\alpha\in\Omega^1_{Ham}(M)$ and $s\geq 1$, set $\alpha_s={N^\ast}^s\alpha$. Then 
\[
[\alpha_i,\alpha_j]_{\pi_{N^k}}=0
\] 
for all $k\geq 0$ and $i,j\geq 1$.
\end{enumerate}
\end{prop}
In other words, $\Omega^1_{Ham}(M)$ is a receptacle for the so called Lenard-Magri, or bi-Hamiltonian, chains.  
An important example is obtained as follows. For every $k\geq 1$, let $I_k:=\frac{1}{2k}\text{Tr}(N^k)$, where $\text{Tr}$ denote the trace operator. Then one can prove that $\{dI_k\}_{k\geq 1}\subset\Omega^1_{Ham}(M)$. More precisely, a direct computation shows that, for every $k\geq 1$, 
\begin{equation}
N^\ast dI_k=dI_{k+1},\label{eq:PNrec}
\end{equation}
which yields the fundamental involutivity property
\begin{equation}
\{I_i,I_j\}=0
\qquad\forall i,j\geq 1,\label{eq:invfam}
\end{equation}
where $\{-,-\}$ is the Poisson bracket defined by $\pi$ via the formula 
$\{f,g\}=\pi(df,dg)$, for all $f,g\in C^\infty(M)$.

\begin{example}[Das-Okubo PN structure \cite{DO}]
\label{ex:1}
On $M=\mathbb R^{2n}$, with coordinates $(p_i,q^i)$, $i=1,\dots,n$, let $\pi=\sum_{i=1}^{n}\frac{\partial}{\partial p_i}\wedge\frac{\partial}{\partial q^i}$ and 
\begin{equation}
\begin{split}
N_{DO}&=\sum_{i=1}^n p_i\Big(\frac{\partial}{\partial q^i}\otimes dq^i
+\frac{\partial}{\partial p_i}\otimes dp_i\Big)+\sum_{i<j}\Big(\frac{\partial}{\partial q^i}\otimes dp_j-\frac{\partial}{\partial q^j}\otimes dp_i\Big)\\
&+\sum_{i=1}^{n-1}e^{q^i-q^{i+1}}\Big(\frac{\partial}{\partial p_{i+1}}\otimes dq^i-\frac{\partial}{\partial p_i}\otimes dq^{i+1}\Big).\label{eq:DO-PN}
\end{split}
\end{equation}
Then $(\pi,N_{DO})$ is a PN structure and the sequence $\{I_k\}_{k\geq 1}$ is a set of first integrals for the \emph{open $n$-particle Toda chain}. In particular, $I_1$ is the total linear momentum and $I_2$ is the energy (Hamiltonian) of this system.
\end{example}

The notion of a Poisson quasi-Nijenhuis structure, introduced in \cite{SX}, represents an interesting generalization of the Poisson-Nijenhuis one. This enhanced setting is defined by a triple $(\pi,N,\phi)$, where $\pi$ is a Poisson tensor, 
$N$ is a tensor field of type (1,1) which is compatible with $\pi$, and $\phi$ is a $d$- and $d_N$-closed 3-form such that 
\begin{equation}
T_N(X,Y)=\pi^\sharp(\phi(X,Y,-))=:\pi^\sharp (i_{X\wedge Y}\phi)
\qquad\forall X,Y\in\mathfrak X(M).\label{eq:PqNtor}
\end{equation}
Since $T_N$ does not vanish,  the corresponding space $\Omega_{Ham}^1(M)$ is not $N^\ast$-invariant and, for this reason, it looses the property of being the natural place where to look for Lenard-Magri chains. In particular, in this more general context, the recurrence equations \eqref{eq:PNrec} assume the form
\begin{equation}
dI_{k+1}=N^\ast dI_k+\phi_{k-1},\qquad\forall k\geq 1,\label{eq:genrec}
\end{equation}
where the 1-forms $\phi_s$ are defined by the formula
\begin{equation}
\langle\phi_s,X\rangle=\frac{1}{2}\text{Tr}(N^s(i_XT_N)),\qquad\forall X\in\mathfrak X(M), s\geq 0, \label{eq:phis}
\end{equation}
and $\langle i_XT_N,Y\rangle=T_N(X,Y)$ for all vector fields $Y$. The recurrence defined in \eqref{eq:genrec} entails that, for all $k>j\geq 1$,
\begin{equation}
\{I_k,I_j\}-\{I_{k-1},I_{j+1}\}=-\langle\phi_{j-1},\pi^\sharp dI_{k-1}\rangle-\langle\phi_{k-1},\pi^\sharp dI_j\rangle,\label{eq:recursionphi}
\end{equation}
which implies that, for a general PqN structure, the family $\{I_k\}_{k\geq 1}$ fails to be 
involutive, i.e., \eqref{eq:invfam} is not satisfied for all possible choices of $i,j$. This observation justifies the following

\begin{defi}\label{def:invpqn}
A PqN structure is called involutive if \eqref{eq:invfam} holds for all $i$ and $j$.
\end{defi}

Even though this definition is quite restrictive, the following example shows that there is at least one important class of PqN structures which meet its requirement.

\begin{example}[PqN structure of the closed Toda system]\label{ex:2}
An important example of involutive PqN structure 
was discovered in \cite{FMOP2020}, see also \cite{FMP2023}, and goes as follows. Let $(M,\pi)$ and $N_{DO}$ be as in Example \ref{ex:1}, and let
\begin{equation}
\label{PqN-Toda}
\widetilde{N}=N_{DO}-e^{q^n-q^1}\Big(\frac{\partial}{\partial p^1}\otimes dq_n-\frac{\partial}{\partial p_n}\otimes dq^1\Big).
\end{equation}
If $\phi=2e^{q^n-q^1}dq^1\wedge dq^n\wedge\sum_{i=1}^ndp_i$, then one can show that $(\pi,\widetilde N,\phi)$ is an involutive PqN structure on $M$. This family of PqN structures describes the so called \emph{closed $n$-particle Toda chain}. More precisely, and in analogy with the case of the Das-Okubo PN structure, the functions ${\tilde I}_1$ and 
${\tilde I}_2$ are, respectively, the total linear momentum and the energy (Hamiltonian) of this classical integrable system. 
\end{example}

It is therefore an interesting problem to pin down a set of assumptions entailing the involutivity of a PqN structure. 
A set of involutivity 
conditions 
was presented in \cite{FMOP2020} and subsequently improved in \cite{FMP2024}. We recall this result in 
\begin{theorem}\label{theo:inv} Let $(\pi,N,\phi)$ be a PqN structure on $M$ and let $I_k=\frac{1}{2k}\text{Tr}\left(N^k\right)$, $k\geq 1$. Suppose there is a 2-form $\Omega$ such that:
\begin{enumerate}
\item[(a)] $\phi=-2\,dI_1\wedge\Omega$;
\item[(b)] $\Omega(X_j,Y_k)=0$ for all $j,k\geq 1$, where $Y_k=N^{k-1}X_1-X_k$ and $X_k=\pi^\sharp dI_k$. 
\end{enumerate}
Then $\{I_j,I_k\}=0$ for all $j,k\ge 1$, where $\{-,-\}$ is the Poisson bracket defined by $\pi$.
\end{theorem}
Another involutivity theorem for PqN manifolds can be found in \cite{FMP2026}.

Before moving to the next section, we make a final comment about the family of PqN structures presented in Example \ref{ex:2}. To this end, we 
recall that 
a problem which is parallel, though independent, to the task of taming the class of involutive PqN structure, 
is the one of ``deforming" a PqN structure into another one. To phrase it in a more precise way, we recall the following theorem  \cite{FMOP2020,DMP2024}.

\begin{theorem}\label{thm:defthe} Let $(\pi,N,\phi)$ be a PqN structure on $M$ and let $\Omega$ be a closed 2-form. If 
\[
\widetilde N=N+\pi^\sharp\Omega^\flat\quad\text{and}\quad\widetilde\phi=\phi+d_N\Omega+
\frac{1}{2}[\Omega,\Omega]_{\pi},
\]
then $(\pi,\widetilde N,\widetilde\phi)$ is a PqN structure. In particular, if $\widetilde\phi=0$, 
the pair $(\pi,\widetilde N)$ is a PN structure.
\end{theorem} 
This result, together with Theorem \ref{theo:inv}, opens the door to the following problem: given a PqN structure, 
which assumptions on $\Omega$ guarantee that the deformed structure $(\pi,\widetilde N,\widetilde\phi)$ is involutive? Or, equivalently, when does a PqN structure admit involutive deformations? Note that, going back to the PN and PqN structures presented in Examples \ref{ex:1} and \ref{ex:2}, one can show that the latter can be obtained as a deformation of the former using $\Omega=e^{q^n-q^1}dq^n\wedge dq^1$, see \cite{FMOP2020,FMP2023,FMP2024}.

\section{Three-dimensional oriented PqN manifolds}\label{sec:3} 

In this section we will work with oriented manifolds, i.e., with pairs $(M,\mathcal V)$ where $M$ is a manifold and $\mathcal V$ is a never vanishing differential form of maximal degree, i.e., a volume form on $M$. Any such form defines a family of isomorphisms $\star_k:\Gamma(\Lambda^k(TM))\rightarrow\Omega^{\dim M-k}(M)$ via 
\[
\star_kP:=i_P\mathcal V,
\] 
that is, the contraction of $\mathcal V$ with $P$, which is the unique $(\dim M-k)$-form such that, for all 
$Q\in\Gamma(\Lambda^{\dim M-k}(M))$,
\[
\langle i_P\mathcal V,Q\rangle=\langle\mathcal V,P\wedge Q\rangle.
\] 
Furthermore, one can associate to every $X\in\mathfrak X(M)$  its divergence, i.e., the unique smooth function $\operatorname{div}(X)$ such that $\mathcal L_X\mathcal V=\operatorname{div}(X)\mathcal V$. Once this has been settled, we can start our study of  
three-dimensional oriented PqN manifolds, stating the following very well known result.

\begin{prop}[\cite{L-GPV}, Proposition 9.20]
\label{prop:xi1} 
Let $(M,\mathcal V)$ be a three-dimensional oriented manifold. 
A bivector field $\pi\in\Gamma(\Lambda^2TM)$ is Poisson if and only if $\xi\wedge d\xi=0$, 
where $\xi=\star_2\pi=i_\pi\mathcal V$.
In this case, 
\begin{equation}
\label{prop:xi2}
\pi^\sharp\xi=0.
\end{equation}
\end{prop}

\begin{remark}
\label{rem:vanish-xi-pi}
Since $\star_2$ is an isomorphism, the Poisson tensor $\pi$ never vanishes if and only if the corresponding 1-form $\xi$ never vanishes.
\end{remark}

We also recall a very nice characterization of three-dimensional oriented PN manifolds, which was presented in \cite{Cheng-Sheng}, Theorem 2.1. 


\begin{theorem}\label{theo:charN}
Let $N$ be a $(1,1)$ tensor field on a three-dimensional oriented Poisson manifold $(M,\pi,\mathcal V)$. Suppose that $\pi$ never vanishes.
Then $(M,\pi,N)$ is a PN manifold if and only if $N$ can be decomposed as 
\begin{equation}
N=\lambda I+Z\otimes\xi,\label{eq:chPN}
\end{equation}
where $\lambda\in C^\infty(M)$, $Z\in\mathfrak X(M)$, $\xi=i_\pi\mathcal V$ and the following conditions hold:
\begin{equation}
d\left(\lambda+\left<\xi,Z\right>\right)=\operatorname{div}(Z)\xi\quad\text{and}\quad Z(\lambda)=0.
\label{eq:contor}
\end{equation}
\end{theorem}

A couple of remarks about the proof of the previous statement are in order.
\begin{remark}\label{rem:2stat}
\begin{enumerate}
\item[]
\item Under the hypotheses of the previous theorem, the condition $N\circ\pi^\sharp=\pi^\sharp\circ N^\ast$ forces the decomposition of the $(1,1)$ tensor field $N$ as $N=\lambda I+Z\otimes\xi$, where $\lambda$, $Z$ and $\xi$ are as in the statement of the theorem, see Lemma 2.4 in \cite{Cheng-Sheng}. On the other hand, 
since $\pi^\sharp\xi=0$ by (\ref{prop:xi2}), if $N=\lambda I+Z\otimes\xi$ then $N\circ\pi^\sharp=\pi^\sharp\circ N^\ast$. For this reason one has that
\[
N=\lambda I+Z\otimes\xi\iff N\circ\pi^\sharp=\pi^\sharp\circ N^\ast.
\]

Furthermore, given the above decomposition of $N$, 
\[
C(\pi,N)\equiv 0 \iff d(\lambda+\langle\xi,Z\rangle)=\operatorname{div}(Z)\xi,
\] 
see Lemma 2.6 in \cite{Cheng-Sheng}. In other words, on an oriented three-dimensional manifold, the pair $(\pi,N)$ is compatible, see \eqref{eq:pncon}, if and only if
\[
N=\lambda I+Z\otimes\xi\quad\text{and}\quad d(\lambda+\langle\xi,Z\rangle)=\operatorname{div}(Z)\xi.
\]
\item On the other hand, in Lemma 2.5 of \cite{Cheng-Sheng} it is proven that if $N$ decomposes as above, then 
\begin{equation}
T_N\equiv 0\iff 
\xi\wedge d(\lambda+\langle\xi,Z\rangle)= 0\quad\text{and}\quad Z(\lambda) = 0.\label{eq:equiv}
\end{equation}
If $\pi$ and $N$ are compatible, then the first identity in the right-hand side of the previous equivalence is automatically verified, so one obtains the proof of Theorem \ref{theo:charN}. Note that  
\eqref{eq:equiv} can be proved using the identity
\begin{equation}\label{eq: T_N computation} 
T_N(X,Y)=\left(\xi\wedge d\left(\lambda+\left<\xi,Z\right>\right)\right)(X,Y)Z+Z(\lambda)i_{\xi}\left(X\wedge Y\right),
\quad\forall X,Y\in\mathfrak X(M),
\end{equation}
see the proof of Lemma 2.5 in \cite{Cheng-Sheng}.
\end{enumerate}
\end{remark}

We are now ready to state the main result of this section which provides 
necessary and sufficient conditions for a triple $(\pi,N,\phi)$, with $\pi$ never vanishing, to be a PqN structure on an oriented three-dimensional manifold $(M,\mathcal V)$.

\begin{theorem}\label{prop 3D PqN} 
Let $(M,\pi)$ be an oriented three-dimensional manifold with volume form $\mathcal V$. 
Suppose that $\pi$ never vanishes.
Then the quadruple $(M,\pi,N,\phi)$ is a PqN manifold if and only if $N$ can be decomposed as $N=\lambda I+Z\otimes\xi$, where $\xi=i_{\pi}\mathcal V$,  $Z\in\mathfrak X(M)$, $\lambda\in C^{\infty}(M)$ and the following conditions hold:
\begin{equation}\label{eq: 3D PqN conditions}
d\left(\lambda+\left<\xi,Z\right>\right)=\operatorname{div}(Z)\xi\quad\text{and}\quad\phi=-Z(\lambda)\mathcal V.
\end{equation}
\end{theorem}

\begin{proof} As seen in the first item of Remark \ref{rem:2stat}, 
the pair $(\pi,N)$ is compatible if and only if
\[
N=\lambda I+Z\otimes\xi\quad\text{and}\quad d(\lambda+\langle\xi,Z\rangle)=\operatorname{div}(Z)\xi.
\]
It remains to prove that $\phi=-Z(\lambda)\mathcal V$. To this end, we first observe that \eqref{eq: T_N computation} yields 
\begin{equation*}
T_N(X,Y)=\cancel{\left(\xi\wedge \operatorname{div}(Z)\xi\right)}(X,Y)Z+Z(\lambda)i_{\xi}(X\wedge Y)
\end{equation*}
i.e.,
\begin{equation}
T_N(X,Y)=Z(\lambda)i_\xi(X\wedge Y).\label{eq: T_N 3D PqN}
\end{equation}
Now, for every $\eta\in\Omega^1(M)$, 
\begin{align*}
\left<\eta,\pi^{\sharp}(i_{X\wedge Y}\phi)\right>&=\pi\left(i_{X\wedge Y}\phi,\eta\right)=\left<\left(i_{X\wedge Y}\phi\right)\wedge \eta,\pi\right>
=\left<\left(i_Y\left(i_X\phi\right)\right)\wedge\eta,\pi\right>\\
&=
\left<i_Y\left(\left(i_X\phi\right)\wedge\eta\right)-\left(i_X\phi\right)\wedge \left(i_Y\eta\right),\pi\right>
=\left<i_Y\left(\left(i_X\phi\right)\wedge\eta\right),\pi\right>-\left<\eta,Y\right>\left<i_X\phi,\pi\right>
\\
&=\left<i_Y\left(i_X\left(\cancel{\phi\wedge\eta}\right)+\phi\wedge\left(i_X\eta\right)\right),\pi\right>-\left<\eta,Y\right>\left<i_X\phi,\pi\right>\\
&=\left<\eta,X\right>\left<i_Y\phi,\pi\right>-\left<\eta,Y\right>\left<i_{\pi}\phi,X\right>=\left<\eta,X\right>\left<i_{\pi}\phi,Y\right>-\left<\eta,Y\right>\left<i_{\pi}\phi,X\right>\\
&=\left<\eta,\left<i_{\pi}\phi,Y\right>X-\left<i_{\pi}\phi,X\right>Y\right>
=-\left<\eta,i_{i_{\pi}\phi}(X\wedge Y)\right>,
\end{align*}
so that 
\begin{equation}
\label{eq:iiphi}
\pi^{\sharp}(i_{X\wedge Y}\phi)=-i_{i_{\pi}\phi}(X\wedge Y).
\end{equation}
Thanks to \eqref{eq: T_N 3D PqN} and \eqref{eq:iiphi}, we have that $T(N)(X,Y)=\pi^{\sharp}(i_{X\wedge Y}\phi)$ 
if and only if $i_{i_{\pi}\phi}(X\wedge Y)=-i_{Z(\lambda)\xi}(X\wedge Y)$. Since $X,Y\in\mathfrak X(M)$ are arbitrary, this is equivalent to 
$$i_{\pi}\phi=-Z(\lambda)\xi.$$ But $\phi=f\mathcal V$ for some $f\in C^\infty(M)$, and 
$\xi=i_\pi\mathcal V$, so we arrive at the condition $(f+Z(\lambda))\xi=0$, which is equivalent to $f=-Z(\lambda)$ since $\xi$ never vanishes --- see Remark \ref{rem:vanish-xi-pi}.
We can therefore conclude that $T_N(X,Y)=\pi^{\sharp}(i_{X\wedge Y}\phi)$ if and only if $\phi=-Z(\lambda)\mathcal V$.
\end{proof}

\begin{remark}\label{rem:indpqn} 
It is worth making a comment about the independence of the PqN structure from the chosen volume form $\mathcal V$. First note that if $\mathcal V'$ is another volume form, then $\mathcal V'=f\mathcal V$ for a given never-vanishing smooth function $f$. This implies the equalities $\xi'=f\xi$, $\lambda'=\lambda$, and $Z'=\frac{1}{f}Z$, so that $d(\lambda+\left<\xi,Z\right>)=d(\lambda'+\left<\xi',Z'\right>)$ and 
${\operatorname{div}}_{\mathcal V}\left(Z\right)\xi={\operatorname{div}}_{\mathcal V'}\left(Z\right)\xi'$, see Remark 2.3 in \cite{Cheng-Sheng}. Finally, 
\[
\phi'=-Z'(\lambda)\mathcal V'=-Z(\lambda)\mathcal V=\phi,
\] 
proving that the PqN-structure is independent of the choice of the volume form.
\end{remark} 

\begin{example}\label{exa:PqN-R3} 
We present a family of PqN structures on $\mathbb R^3$, to be compared with the corresponding family of PN 
structures in \cite{Cheng-Sheng}, Proposition 3.1. Let $\pi=\frac{\partial}{\partial x}\wedge\frac{\partial}{\partial y}\in\Gamma(\Lambda^2 T\mathbb R^3)$ 
and let $\mathcal V=dx\wedge dy\wedge dz$ be the canonical volume form on $\mathbb R^3$. 
If $N$ is a (1,1) tensor field and $\phi$ is a 3-form, we know from 
Theorem \ref{prop 3D PqN}
that the quadruple $(\mathbb R^3,\pi,N,\phi)$ is a PqN manifold if and only if 
$N=\lambda I+Z\otimes\xi$, where $\xi=i_{\pi}\mathcal V=dz$, and the following conditions hold:
\begin{equation*}
d\left(\lambda+\left<\xi,Z\right>\right)=\operatorname{div}(Z)\xi\quad\text{and}\quad\phi=-Z(\lambda)\mathcal V.
\end{equation*}
If 
\begin{equation}
\label{eq: Z}
Z=a\frac{\partial}{\partial x}+b\frac{\partial}{\partial y}+c\frac{\partial}{\partial z},
\end{equation}
then 
\begin{align}
d(\left<\xi,Z\right>+\lambda)=
\operatorname{div}(Z)
\xi\ \Longleftrightarrow\ (c+\lambda)_xdx+(c+\lambda)_ydy+(c+\lambda)_zdz
=(a_x+b_y+c_z)dz,\label{eq:equico}
\end{align}
which is equivalent to
\begin{align*}
\begin{cases}
(c+\lambda)_x=0\\
(c+\lambda)_y=0\\
\lambda_z=a_x+b_y.
\end{cases}
\end{align*}
Thus $c+\lambda$ is a function of $z$ only, say, $g(z)$. 
Since $Z(\lambda)=a\lambda_x+b\lambda_y+c\lambda_z$, one has that
\begin{equation}
\phi=-Z(\lambda)\mathcal V=(ac_x+bc_y-c(a_x+b_y))dx\wedge dy\wedge dz.\label{eq:phi}
\end{equation}
Hence we found a recipe to complete $\pi=\frac{\partial}{\partial x}\wedge\frac{\partial}{\partial y}$ 
to a PqN structure on $\mathbb R^3$ endowed with the volume form $\mathcal V=dx\wedge dy\wedge dz$. 
One has to choose arbitrary functions $\lambda,
a\in C^{\infty}(\mathbb{R}^3)$ and $g(z)$. Then, after defining
\[
b=\int({\lambda}_z-a_x)dy\quad\text{and}\quad c=g-\lambda,
\]
one constructs the vector field (\ref{eq: Z}) and the PqN structure $(\mathbb{R}^3,\pi,N,\phi)$, where $N=\lambda I+Z\otimes dz$ and $\phi$ is given by \eqref{eq:phi}.
\end{example}

A first application of Theorem \ref{prop 3D PqN} focuses on the problem of deformed PqN structures in the sense of Theorem \ref{thm:defthe}. 
\begin{prop}\label{prop:local}
Locally, every PqN structure on an oriented three-dimensional manifold with a never vanishing Poisson bivector is obtained by deforming a PN structure by means of a closed 2-form. 
\end{prop}
\begin{proof}
Let $(\pi,N,\phi)$ be a PqN structure, with $\pi$ never vanishing, on an oriented three-dimensional manifold  $\left(M,\mathcal V\right)$ and let $m\in M$. Darboux theorem entails that there exists a coordinate neighborhood $U$ of  $m$ with coordinates, say $(x,y,z)$, such that $\pi=\frac{\partial}{\partial x}\wedge\frac{\partial}{\partial y}$ on $U$. 
Moreover, one can consider $\mathcal V=dx\wedge dy\wedge dz$ on $U$, see Remark \ref{rem:indpqn}. Then, on $U$, Example \ref{exa:PqN-R3} shows that $N=\lambda I+Z\otimes \xi$ where $\xi=dz$ and $Z$ is given by \eqref{eq: Z} verifying $\lambda+c=g(z)$ and $\lambda_z=a_x+b_y$.  Now consider
\[
\Omega=\Omega_{12}dx\wedge dy+\Omega_{13}dx\wedge dz+\Omega_{23}dy\wedge dz
=-\lambda\, dx\wedge dy-b\,dx\wedge dz+a\,dy\wedge dz,
\]
and note that $d\Omega=0$ since $\lambda_z=a_x+b_y$. Using $\lambda+c=g(z)$, 
one can compute
\begin{align*}
\pi^{\sharp}\Omega^{\flat}
&=-\Omega_{12} I+\left(\Omega_{23}\frac{\partial}{\partial x}-\Omega_{13}\frac{\partial}{\partial y}+\Omega_{12}\frac{\partial}{\partial z}\right)\otimes dz\\
&=\lambda I+\left(a\frac{\partial}{\partial x}+b\frac{\partial}{\partial y}-(g(z)-c)\frac{\partial}{\partial z}\right)\otimes dz\\
&=\lambda I+
Z\otimes dz-g(z)\frac{\partial}{\partial z}\otimes dz.
\end{align*}
If $\widetilde{Z}:=g(z)\frac{\partial}{\partial z}$ and $N_1:=\widetilde{Z}\otimes dz$, 
then 
$N=N_1+\pi^{\sharp}\Omega^{\flat}$.
One can observe that $N_1$ is a $(1,1)$ tensor field compatible with $\pi=\frac{\partial}{\partial x}\wedge\frac{\partial}{\partial y}$, since $N_1\circ\pi^{\sharp}=\pi^{\sharp}\circ N_1^{*}$ and
\[
d\langle \xi,\widetilde{Z}\rangle=d(g(z))=g'(z)dz=\operatorname{div}(\widetilde{Z})\xi,
\]
see Remark \ref{rem:2stat}. Moreover, \eqref{eq:equiv} shows that $T_{N_1}=0$, proving that $(\pi,N_1)$ is a PN structure. 
To conclude that the original PqN structure is a deformation of this PN structure, we are left to show 
\begin{equation}
d_{N_1}\Omega+\frac{1}{2}\left[\Omega,\Omega\right]_{\pi}=\phi=-Z(\lambda)\mathcal V,\label{eq:dNphi}
\end{equation}
whose computational details are enclosed in Section \ref{subsec:app1}.
\end{proof}

Another significant application of Theorem \ref{prop 3D PqN}  is its use in the following
\begin{prop}\label{thm:inv3}
Every three-dimensional oriented PqN manifold with a never vanishing Poisson tensor is involutive.
\end{prop}
\begin{proof}
Using local Darboux coordinates $(x,y,z)$ and the notation as in Example \ref{exa:PqN-R3}, one has
\begin{equation}
N=\begin{bmatrix}\lambda(x,y,z) & 0 & a(x,y,z)\\ 0 & \lambda(x,y,z) & b(x,y,z) \\ 0 & 0 & g(z)\end{bmatrix}.\label{eq:thanksRef}
\end{equation}
This implies that $\text{Tr} N^k=2\lambda^k+g^k$ for every $k\geq 1$. Since $\pi=\frac{\partial}{\partial x}\wedge\frac{\partial}{\partial y}$, 
one has that 
\[
\{I_i,I_j\}=\pi(dI_i,dI_j)=0
\]
for all $i,j$.\end{proof}

For an alternative proof of this proposition which does not use local coordinates, we refer the reader to Subsection \ref{ss:proofth}.

\section{PqN and Haantjes structures}\label{sec:4}

We start this section recalling the definition of Haantjes torsion of a tensor field of type $(1,1)$, then we introduce the notion of Haantjes manifold and we discuss how this relates to the geometry of PqN manifolds. 

Let $N$ be a (1,1) tensor field. Recall that its Nijenhuis torsion $T_N$ is the tensor field of type $(1,2)$, i.e., the 2-form with values in $TM$, defined in \eqref{eq:torsionA}. To address the problem of the Frobenius integrability of 2-dimensional planes generated by the eigenvectors of $N$ (see Subsection \ref{ss:haantjes} 
and the references therein), 
Haantjes introduced in \cite{Haantjes} what is now known as the \emph{Haantjes torsion} of $N$. This is the tensor field 
$H_N$ of type $(1,2)$ defined by the formula
\begin{equation}
H_N(X,Y)=T_N(NX,NY)-N(T_N(NX,Y)+T_N(X,NY)-NT_N(X,Y)),\qquad
X,Y\in\mathfrak X(M), \label{eq:haant}
\end{equation}
where $T_N$ is the Nijenhuis torsion of $N$. The tensor $N$ will be called a Haantjes operator if $H_N\equiv 0$.

In \cite{BogoTens}, Bogoyavlenskij derived algebraic identities for Nijenhuis and Haantjes torsions on an arbitrary manifold. some of these results, which will be used subsequently, are collected in the following
\begin{lem}[\cite{BogoTens}, Corollaries 3, 4]\label{lem:Bogo}
Let  $N$ be a (1,1) tensor field on a 
manifold $M$ and let $p(x,t)=\sum_{m=0}^ka_m(x)t^m$ be a polynomial with coefficients in $C^\infty(M)$.  Then: 
\begin{enumerate}
\item If $N$ is a Haantjes operator, $p(x,N)$ is also a Haantjes operator.
\item For all $f,\,g\in C^{\infty}(M)$
\begin{equation}
H_{fI+gN}(X,Y)=g^4H_{N}(X,Y).\label{eq: Haantjes_defor_I}
\end{equation}
\end{enumerate}
\end{lem}
The next example presents a formula for the Haantjes torsion of 
the tensor product of a vector field and a 1-form.
\begin{example} 
Let $N=W\otimes\eta$, where $W\in\mathfrak{X}(M)$ and $\eta\in\Omega^{1}(M)$. Then for all $X,Y\in\mathfrak{X}(M)$
\begin{equation}
T_{N}(X,Y)=\left[\left<\eta,X\right>Y(\left<\eta,W\right>)-\left<\eta,Y\right>X(\left<\eta,W\right>)-(\eta\wedge d\eta)(W,X,Y)\right]W,\label{eq: NtorsionWxn}
\end{equation}
see Lemma 2.5 in \cite{Cheng-Sheng}.
Moreover a computation using \eqref{eq: NtorsionWxn} shows that
\begin{align*}
T_N(NX,NY)&=0,\\
NT_N(NX,Y)&=\left<\eta,W\right>\left[\left<\eta,X\right>\left<\eta,W\right>Y(\left<\eta,W\right>)-\left<\eta,Y\right>\left<\eta,X\right>W(\left<\eta,W\right>)\right]W,\\
NT_N(X,NY)&=\left<\eta,W\right>\left[\left<\eta,X\right>\left<\eta,Y\right>W(\left<\eta,W\right>)-\left<\eta,Y\right>\left<\eta,W\right>X(\left<\eta,W\right>)\right]W,\\
N^2T_N(X,Y)&=\left<\eta,W\right>^2\left[(\left<\eta,X\right>Y(\left<\eta,W\right>)-\left<\eta,Y\right>X(\left<\eta,W\right>)-(\eta\wedge d\eta)(W,X,Y)\right]W.
\end{align*}
Summing up the above four equations, \eqref{eq:haant} yields
\begin{equation}\label{eq: HtorsionWxn}
H_{N}(X,Y)=-\left<\eta,W\right>^2(\eta\wedge d\eta)(W,X,Y)W.
\end{equation}
\end{example}

As we already observed, the presence of a PN or, more generally, of a bi-Hamiltonian structure on $M$, makes it possible to introduce a class of distinguished vector fields (called bi-Hamiltonian), which possess a large number of conserved quantities which are mutually in involution with respect to the underlying Poisson structures. In particular, in the PN case, the presence of the recursion operator defines a natural space where to look for the Lenard-Magri chains. In the attempt to weaken the hypothesis on the recursion operator to be torsion free, but still keeping the possibility to define a notion of Lenard-Magri chain, Magri introduced in \cite{MagriHaa} the notion of Haantjes manifold, which we recall in 

\begin{defi} A Haantjes structure on a manifold $M$ is a pair $(N,\theta)$ of a $(1,1)$ tensor field $N$ 
and a 1-form $\theta$ such that:
\begin{itemize}
\item[(H1)] $H_N=0$;
\item[(H2)] $d\theta=0$;
\item[(H3)] $d_N\theta=0$;
\item[(H4)] $\langle\theta,T_N(X,Y)\rangle=0$ for all $X,Y\in\mathfrak X(M)$.
\end{itemize}
A manifold endowed with a Haantjes structure is called a Haantjes manifold. If $\theta=0$, the corresponding Haantjes structures is called trivial.
\end{defi}

We can now state and prove the main result of this section. 
\begin{theorem}
\label{theo:haan} 
If $(\pi,N,\phi)$ is a PqN structure on an oriented three-dimensional manifold $(M,\mathcal V)$ and $\pi$ never vanishes, then  $M$ admits a Haantjes structure $(N,\theta)$. 
\end{theorem}
\begin{proof} 
We know from Theorem \ref{prop 3D PqN} that 
$N=\lambda I+Z\otimes\xi$, where $\xi=i_{\pi}\mathcal{V}$, $\lambda\in C^{\infty}(M)$ and $Z\in\mathfrak X(M)$ are related by the first equation of \eqref{eq: 3D PqN conditions}. We claim that the pair $(N,\theta)$, where 
\begin{equation}\label{eq: 1-form theta}
\theta=\text{div}(Z)\xi,
\end{equation}
is a Haantjes structure on $M$. 
The exactness of $\theta$, see \eqref{eq: 3D PqN conditions}, gives us (H2). Now observe that, for every $X\in\mathfrak X(M)$,
\[
i_N\theta(X)=\langle\theta,NX\rangle=\langle\theta,\lambda X+\langle\xi,X\rangle Z\rangle=\lambda\langle\theta,X\rangle+\langle\theta,Z\rangle\langle\xi,X\rangle=(\lambda+\langle\xi,Z\rangle)\langle\theta,X\rangle,
\]
which yields the identity
\begin{equation}
i_N\theta=(\lambda+\langle\xi,Z\rangle)\theta.\label{eq:iNxi}
\end{equation}
Since $d\theta=0$, one has that
\[
d_N\theta=-d(i_N\theta)\stackrel{\eqref{eq:iNxi}}{=}-d((\lambda+\langle\xi,Z\rangle)\theta)=-d(\lambda+\langle\xi,Z\rangle)\wedge\text{div}(Z)\xi\stackrel{\eqref{eq: 3D PqN conditions}}{=}-\left(\text{div}(Z)\right)^2(\xi\wedge\xi)=0,
\]
proving (H3). Moreover, 
\[
\langle\theta,T_N(X,Y)\rangle\stackrel{\eqref{eq: T_N 3D PqN}}{=}Z(\lambda)\text{div}(Z)\langle\xi,i_\xi(X\wedge Y)\rangle=Z(\lambda)\text{div}(Z)\langle\xi,\langle\xi,X\rangle Y-\langle\xi,Y\rangle X\rangle=0
\]
for all $X,Y\in\mathfrak X(M)$, proving (H4). We are left to prove (H1). 
Applying \eqref{eq: Haantjes_defor_I} to $N=\lambda I+Z\otimes\xi$ and using \eqref{eq: HtorsionWxn}, we obtain
\[
H_N(X,Y)=H_{Z\otimes\xi}(X,Y)=-\left<\xi,Z\right>^2\left(\xi\wedge d\xi\right)(Z,X,Y)Z,
\]
which vanishes identically, see Proposition \ref{prop:xi1}.
\end{proof}

\begin{remark}
\begin{enumerate}
\item[]
\item In the same hypotheses of the previous theorem, Remark \ref{rem:vanish-xi-pi} implies that 
the Haantjes structure $(N,\theta)$ is non-trivial if the divergence of the vector field $Z$ is not zero.
\item For a general PqN structure $(\pi,N,\phi)$, the Haantjes torsion of $N$ does not vanish. An example is given by the case of the periodic Toda lattice, see \eqref{PqN-Toda}.
\end{enumerate}
\end{remark}

\begin{remark}
\label{rem:Eber}
A Haantjes structure $(N,\theta)$ on an oriented three-dimensional manifold cannot always be derived from a PqN structure $(N,\pi,\phi)$, since a non-trivial Poisson bivector $\pi$ such that  $N\circ\pi^{\sharp}=\pi^{\sharp}\circ N^*$ does not always exist. Indeed, consider 
(see Example 2 in \cite{MagriVeselov}) the (1,1) tensor field
\[
N=-\frac{y}{2}\frac{\partial}{\partial x}\otimes dz+2\frac{\partial}{\partial y}\otimes dx-z\frac{\partial}{\partial y}\otimes dz+2\frac{\partial}{\partial z}\otimes dy.
\]
It can be checked that its Haantjes torsion vanishes and that its Nijenhuis torsion
\[
T_N=\frac{\partial}{\partial x}\otimes dz\otimes dx-\frac{\partial}{\partial x}\otimes dx\otimes dz+\frac{\partial}{\partial y}\otimes dz\otimes dy-\frac{\partial}{\partial y}\otimes dy\otimes dz
\]
satisfies $\left<dz,T_N(X,Y)\right>=0$ for all $X,Y\in \mathfrak{X}(\mathbb R^3)$. Moreover,  $d_N(dz)=-d(i_Ndz)=0$ since $i_Ndz=2dy$. 
Therefore the pair $(N,dz)$ is a Haantjes structure on $\mathbb R^3$. 
However, a simple computation shows that a non-trivial Poisson bivector $\pi$ such that 
$\pi^{\sharp}\circ N^*=N\circ \pi^{\sharp}$ does not exist.
This entails that the Haantjes structure $(N,dz)$ cannot be derived from a PqN structure.
\end{remark}

\subsection{Generalized Lenard-Magri chains}

Haantjes manifolds were introduced by Magri in \cite{MagriHaa} with the aim of generalizing the notion of Lenard-Magri chain, see Definition \ref{def:gen-LM} below. In this more general geometrical setting, the sequence of Nijenhuis torsionless tensors 
$\{N^i\}
$, appearing in the PN case, is traded for a family of $(1,1)$ tensors with vanishing Haantjes torsion. More precisely,

\begin{defi}[Generalized Lenard-Magri chains]\label{def:genchain} 
\label{def:gen-LM}
A generalized Lenard-Magri chain of length $p$ on a Haantjes manifold $M$ with Haantjes structure $(N,\theta)$ is a family  
$\mathcal N=\{N_0,N_1,\dots,N_{p-1}\}
$ of $(1,1)$ tensor fields satisfying, for all $i,j= 0,\dots,p-1$:
\begin{itemize}
\item[(C0)] $N_0=I$ and $N_1=N$;
\item[(C1)] $H_{N_i}=0$; 
\item[(C2)] $N_i\circ N_j=N_j\circ N_i$; 
\item[(C3)] $d_{N_i}\theta=0$; 
\item[(C4)] $\langle\theta,T_{N_i}(X,Y)\rangle=0$ 
for all $X,Y\in\mathfrak X(M)$. 
\end{itemize}
\end{defi}

The relevance of these conditions relies on the remarkable observation that, on a Haantjes manifold endowed with a generalized Lenard-Magri chain, the 1-forms $\theta_{ij}=N_iN_j\theta$ are closed --- see Proposition 1 in \cite{MagriHaa}.

\begin{remark}
The name generalized Lenard-Magri chain is not universally adopted. In particular, in \cite{MagriHaa} the structure introduced in Definition \ref{def:genchain} is simply called a Lenard chain. 
\end{remark} 

\begin{example}
Considering the same (1,1) tensor field $N$ as in Remark \ref{rem:Eber}, we notice that
\[
N^2=4\frac{\partial}{\partial z}\otimes dx-\left(y\frac{\partial}{\partial x}+2z\frac{\partial}{\partial y}\right)\otimes dy-\left(y\frac{\partial}{\partial y}+2z\frac{\partial}{\partial z}\right)\otimes dz
\]
has Nijenhuis torsion
\[
T_{N^2}=-8\frac{\partial}{\partial y}\otimes dx\otimes dy+8\frac{\partial}{\partial y}\otimes dy\otimes dx-8\frac{\partial}{\partial z}\otimes dx\otimes dz+8\frac{\partial}{\partial z}\otimes dz\otimes dx.
\]
Thus $\mathcal N=\{I,N,N^2\}$ is not a generalized Lenard-Magri chain since $\left<dz,T_{N^2}(X,Y)\right>$ does not vanish for all $X,Y\in\mathfrak{X}(\mathbb R^3)$.
\end{example}

Let $(N,\theta)$ be the Haantjes structure, with $\theta$ given by \eqref{eq: 1-form theta}, 
derived from a PqN-structure $(\pi,N,\phi)$ on an oriented three-dimensional manifold $M$ 
with volume form $\mathcal{V}$, see Theorem \ref{theo:haan}. 
\begin{theorem}
\label{theo:Len-Ma-Ge} 
The family $\mathcal N=\{N^0,N^1,\dots,N^{p-1}\}$
is a generalized Lenard-Magri chain associated to the Haantjes structure $(N,\theta)$.
\end{theorem} 
\begin{proof}
To prove the statement, one needs to check that conditions (C1), (C3) and (C4) are satisfied. To this end, first observe that Lemma \ref{lem:Bogo}-1 entails that  $H_{N^k}=0$ for every $k\geq 2$, since $H_N=0$.  
As far as (C3) is concerned, 
we notice that, for every $X\in\mathfrak X(M)$,
\[
\langle i_{N^k}\theta,X\rangle=\langle\theta,N^kX\rangle=\langle i_{N}\theta,N^{k-1}X\rangle\stackrel{\eqref{eq:iNxi}}{=}(\lambda+\langle\xi,Z\rangle)\langle i_{N^{k-1}}\theta,X\rangle,
\]
which entails
\begin{equation}
i_{N^k}\theta=(\lambda+\langle\xi,Z\rangle)^k\theta.\label{eq:iNxi1}
\end{equation}
Since $\theta$ is closed 
\[
d_{N^k}\theta=-di_{N^k}\theta=-d(\lambda+\langle\xi,Z\rangle)^k\theta=-k(\lambda+\langle\xi,Z\rangle)^{k-1}d(\lambda+\langle\xi,Z\rangle)\wedge\theta\stackrel{\eqref{eq: 3D PqN conditions}}{=}0.
\]
To prove that (C4) holds, we simply observe that $\left<\theta,T_{N^k}(X,Y)\right>$ vanishes because of (\ref{eq: T_N 3D PqN-k}) and the definition of $\theta$.
\end{proof}

\subsection{Comparison with the work of Tempesta and Tondo}\label{ss:tempton}
We will now make a few comments to link the theory presented in the previous sections with the recent works of Tempesta and Tondo, see \cite{tempestatondoben, tempestatondo, tempestatondohigher,tempestatondoclass} and \cite{TondoLag}. We start recalling 
the notion of Haantjes algebra, 
introduced in \cite{tempestatondo}, 
Definition 36. 
This consists of a set of $(1,1)$ tensor fields $\{N_i\}_{i\in I}$ on a manifold $M$ 
satisfying the following two properties:
\begin{itemize}
\item $H_{fN_i+gN_j}=0$, for all $i,j\in I$ and $f,g\in C^\infty(M)$;
\item $H_{N_iN_j}=0=H_{N_jN_i}$, for all $i,j\in I$.
\end{itemize}

A Haantjes algebra will be called of rank $m$ if the corresponding $C^\infty(M)$-module has rank equal to $m$. The Haantjes algebra will be called abelian if the operators $N_i$ 
pairwise commute. An important class of abelian Haantjes algebras are the cyclic ones, i.e., the Haantjes algebras generated by a given Haantjes operator. A cyclic Haantjes algebra has rank $m$ if the minimal polynomial of its generator has degree $m$. 

After these preliminary comments, we easily obtain 
\begin{prop}\label{pro:TTpqn}
If $(\pi,N,\phi)$ is a PqN structure on an oriented three-dimensional manifold and $\pi$ never vanishes, 
then $N$ generates a rank-two 
cyclic Haantjes algebra.
\end{prop}

\begin{proof}
As we already noticed in Theorem \ref{theo:haan}, $N$ is a Haantjes operator.  Then $\lbrace N^k\rbrace_{k\geq 0}$ is a cyclic Haantjes algebra as a consequence of Lemma \ref{lem:Bogo}-1. Moreover, the rank of such algebra is two since the minimal polynomial of $N=\lambda I+Z\otimes\xi$ is  $p(t)=t^2-(\left<\xi,Z\right>+2\lambda)t+\lambda^2+\left<\xi,Z\right>\lambda$.
\end{proof}

Given a Haantjes algebra of rank $m$ and a set $N_1,\dots,N_m$ of its generators (as a $C^\infty(M)$-module), a closed 1-form $\alpha$ generates a so called \emph{Haantjes chain} $(N_1^\ast\alpha,\dots,N_m^\ast\alpha)$ of length $m$ if
$
d_{N_k}\alpha=0,
$
i.e., if $N^\ast_k\alpha=i_{N_k}\alpha$ is closed for all $k=1,\dots,m$, and the forms $N_k^\ast\alpha$, with $k=1,\dots,m$, are linearly independent, see Definition 15 in \cite{tempestatondoclass}.

If $(\pi,N,\phi)$ is as in Proposition \ref{pro:TTpqn}, where $N=\lambda I+Z\otimes\xi$ 
and $\theta$ given as in \eqref{eq: 1-form theta},
then the 1-forms 
$\theta_k={N^k}^{\ast}\theta=i_{N^k}\theta$, 
see \eqref{eq:iNxi1}, are closed but, clearly, they are not independent (since 
$\theta$ is a common eigenvector of the ${N^k}^\ast$) and for this reason they fail to form a Haantjes chain.

Finally, it is worth mentioning that every PqN structure satisfying the hypothesis of Proposition \ref{pro:TTpqn} is a \emph{Poisson-Haantjes structure} in the sense of Definition 6 in \cite{TondoLag}.

\section{Loose ends}\label{sec:Appendix}

We divide this section in three parts. The first one is devoted to the proof of \eqref{eq:dNphi},
the second one to an alternative proof of 
Proposition \ref{thm:inv3}  and the last one is aimed to collect some more information about the Haantjes torsion.

\subsection{Proof of the identity \eqref{eq:dNphi}}\label{subsec:app1}
First notice that, given $\pi=\frac{\partial}{\partial x}\wedge\frac{\partial}{\partial y}$  and any closed 2-form 
$\Omega=\Omega_{12}dx\wedge dy+\Omega_{13}dx\wedge dz+\Omega_{23}dy\wedge dz$, we have
\[
\frac{1}{2}\left[\Omega,\Omega\right]_{\pi}=\left(\Omega_{23}\frac{\partial\Omega_{12}}{\partial x}-\Omega_{13}\frac{\partial\Omega_{12}}{\partial y}+\Omega_{12}\frac{\partial\Omega_{12}}{\partial z}\right)\mathcal V.
\]
In addition, for $N_1=\widetilde{Z}\otimes dz$ as considered in 
Proposition \ref{prop:local}, one has that
\begin{equation}\label{eq: d_N1Omega} 
d_{N_1}\Omega=-di_{N_1}\Omega=-d\left(i_{N_1}(\Omega_{12}dx\wedge dy)+i_{N_1}(\Omega_{13}dx\wedge dz)+i_{N_1}(\Omega_{23}dy\wedge dz)\right).
\end{equation}
We observe that
\begin{align*}
i_{N_1}(\Omega_{12}dx\wedge dy)(X,Y)
 &=(\Omega_{12}dx\wedge dy)\left(g(z)dz(X)\frac{\partial}{\partial z},Y\right)+(\Omega_{12}dx\wedge dy)\left(X,g(z)dz(Y)\frac{\partial}{\partial z}\right)=0,
\end{align*}
\begin{align*}
i_{N_1}(\Omega_{13}dx\wedge dz)(X,Y)
&=(\Omega_{13}dx\wedge dz)\left(g(z)dz(X)\frac{\partial}{\partial z},Y\right)+(\Omega_{13}dx\wedge dz)\left(X,g(z)dz(Y)\frac{\partial}{\partial z}\right)\\
&=g(z)\Omega_{13}dx(X)dz(Y)-g(z)\Omega_{13}dx(Y)dz(X)\\
&=g(z)(\Omega_{13}dx\wedge dz)(X,Y),\\
i_{N_1}(\Omega_{23}dy\wedge dz)(X,Y)
&=(\Omega_{23}dy\wedge dz)\left(g(z)dz(X)\frac{\partial}{\partial z},Y\right)+(\Omega_{23}dy\wedge dz)\left(X,g(z)dz(Y)\frac{\partial}{\partial z}\right)\\
&=g(z)\Omega_{23}dy(X)dz(Y)-g(z)\Omega_{23}dy(Y)dz(X)\\
&=g(z)(\Omega_{23}dy\wedge dz)(X,Y).
\end{align*}
By substituting the three preceding equations into \eqref{eq: d_N1Omega}, we obtain
\begin{align*}
d(i_{N_1}\Omega)&=d(g(z)\Omega-g(z)\Omega_{12}dx\wedge dy)=dg(z)\wedge \Omega-\frac{\partial \left(g(z)\Omega_{12}\right)}{\partial z}\mathcal V\\
&=g(z)'dz\wedge (\Omega_{12}dx\wedge dy+\Omega_{13}dx\wedge dz+\Omega_{23}dy\wedge dz)-\left(g(z)'\Omega_{12}+g(z)\frac{\partial \Omega_{12}}{\partial z}\right)\mathcal V\\
&=-g(z)\frac{\partial \Omega_{12}}{\partial z}\mathcal V.
\end{align*}
Thus if $\Omega$ is a closed 2-form on $\mathbb{R}^3$, then 
\begin{equation}\label{eq: deformed_part}
d_{N_1}\Omega+\frac{1}{2}\left[\Omega,\Omega\right]_{\pi}=\left(-g(z)\frac{\partial \Omega_{12}}{\partial z}+\Omega_{23}\frac{\partial\Omega_{12}}{\partial x}-\Omega_{13}\frac{\partial\Omega_{12}}{\partial y}+\Omega_{12}\frac{\partial\Omega_{12}}{\partial z}\right)\mathcal V.
\end{equation}
In particular, the 2-form $\Omega$ given in 
Proposition \ref{prop:local} is closed and from \eqref{eq: deformed_part} we derive 
\[
d_{N_1}\Omega+\frac{1}{2}\left[\Omega,\Omega\right]_{\pi}=(g(z)\lambda_z-a\lambda_x-b\lambda_y+\lambda\lambda_z)\mathcal V=-Z(\lambda)\mathcal V=\phi,
\]
since $\Omega_{12}=-\lambda,\Omega_{13}=-b,\Omega_{23}=a$ and $\lambda+c=g(z)$.


\subsection{An alternative proof of Proposition \ref{thm:inv3}}\label{ss:proofth} In this subsection we present an alternative proof of Proposition \ref{thm:inv3} which does not use local coordinates. In spite of being more involved, it is interesting per-se and we think it could play a role in a possible extension of the results of this paper to higher (greater than 3) dimension.

The proof is based on a closed formula for the 1-forms $\phi_k$, see \eqref{eq:genrec}, which is contained in the following
\begin{lem}\label{lem:phik}
	On an oriented three-dimensional PqN manifold, where $N=\lambda I+Z\otimes\xi$, we have that 
	\begin{equation}
		\phi_k=Z(\lambda)\lambda^k\xi\qquad\mbox{for all $k\geq 0$.}
		\label{eq:phik}
	\end{equation}
\end{lem}
\begin{proof}
	The first step is to prove 
	that 
	\begin{equation}
		N^k=\lambda^kI+f_k \,Z\otimes\xi,\label{eq:Nalk}
	\end{equation}
	where $f_k=\sum_{l=0}^{k-1}{k\choose l}\lambda^l\langle\xi,Z\rangle^{k-l-1}$. 
The proof is done by induction, noticing that for $k=1$ one has $N=\lambda I+Z\otimes \xi$, which is the content of Theorem \ref{theo:charN}.
Suppose now that \eqref{eq:Nalk} holds for $k$ and let us prove it for $k+1$. We compute
\begin{equation}
\label{eq:s1}
\begin{aligned}
N^{k+1}(X)&=(\lambda I+Z\otimes\xi)\Big(\lambda^kX+\sum_{l=0}^{k-1}{k\choose l}\lambda^l\langle\xi,Z\rangle^{k-l-1}\langle\xi,X\rangle Z\Big)
\\
&=\lambda^{k+1}X+\underbrace{\sum_{l=0}^{k-1}{k\choose l}\lambda^{l+1}\langle\xi,Z\rangle^{k-l-1}\langle\xi,X\rangle Z}_{(1)}+\underbrace{\sum_{l=0}^{k-1}{k\choose l}\lambda^{l}\langle\xi,Z\rangle^{k-l}\langle\xi,X\rangle Z+\lambda^k\langle\xi,X\rangle Z}_{(2)}.
\end{aligned}
\end{equation}
Now we note that in \eqref{eq:s1} one has
\begin{eqnarray}
&(1)&\stackrel{s=l+1}{=}\sum_{s=1}^{k}{k\choose s-1}\lambda^{s}\langle\xi,Z\rangle^{k-s}\langle\xi,X\rangle Z,\label{eq:s2}\\
&(2)&=\sum_{l=0}^k{k\choose l}\lambda^l\langle\xi,Z\rangle^{k-l}\langle\xi,X\rangle Z=\sum_{l=1}^k{k\choose l}\lambda^l\langle\xi,Z\rangle^{k-l}\langle\xi,X\rangle Z+\langle\xi,Z\rangle^k\langle\xi,Z\rangle Z.\label{eq:s3}
\end{eqnarray}
By renaming $l$ the running variable $s$ in \eqref{eq:s2}, plugging \eqref{eq:s2} and \eqref{eq:s3} in \eqref{eq:s1}, 
and using the identity ${k\choose l-1}+{k\choose l}={k+1\choose l}$, one obtains 
\begin{eqnarray*}
N^{k+1}(X)
&=&\lambda^{k+1}X+\sum_{l=1}^{k}\Big[{k\choose l-1}+{k\choose l}\Big]\lambda^{l}\langle\xi,Z\rangle^{k-l}\langle\xi,X\rangle Z+\langle\xi,Z\rangle^k\langle\xi,Z\rangle Z\\
&
=&\lambda^{k+1}X+\sum_{l=1}^{k}{k+1\choose l}\lambda^{l}\langle\xi,Z\rangle^{k-l}\langle\xi,X\rangle Z+\langle\xi,Z\rangle^k\langle\xi,Z\rangle Z\\
&=&\lambda^{k+1}X+\sum_{l=0}^{k}{k+1\choose l}\lambda^{l}\langle\xi,Z\rangle^{k-l}\langle\xi,X\rangle Z,
\end{eqnarray*}
which, being $X$ arbitrary, entails the identity \eqref{eq:Nalk}.  

The second step starts with the remark that, since
	\[
	T_N(X,Y)\stackrel{\eqref{eq: T_N 3D PqN}}{=}Z(\lambda)i_\xi(X\wedge Y)=Z(\lambda)\langle\xi,X\rangle Y-Z(\lambda)\langle\xi,Y\rangle X,
	\]
	one has that
	\begin{equation}
		i_XT_N=Z(\lambda)(\langle\xi,X\rangle I-X\otimes\xi).\label{eq:iXTN}
	\end{equation}
	For every $k\geq 0$ and $X\in\mathfrak X(M)$, we can now compute
	\[
	\begin{split}
		\langle\phi_k,X\rangle&\stackrel{
			\eqref{eq:phis}}{=}\frac12\text{Tr}(N^ki_XT_N)\stackrel{\eqref{eq:iXTN}}
		{=}\frac12\text{Tr}\big(N^k(Z(\lambda)\langle\xi,X\rangle I-Z(\lambda)X\otimes\xi)\big)\\
		&\stackrel{\eqref{eq:Nalk}}{=}\frac12\text{Tr}\Big[\Big(\lambda^kI+\sum_{l=0}^{k-1}{k\choose l}\lambda^l\langle\xi,Z\rangle^{k-l-1}Z\otimes\xi \Big)\Big(Z(\lambda)\langle\xi,X\rangle I-Z(\lambda)X\otimes\xi)\Big)\Big]\\
		&=\frac12\text{Tr}\Big[Z(\lambda)\lambda^k\langle\xi,X\rangle I-Z(\lambda)\lambda^kX\otimes\xi+Z(\lambda)\langle\xi,X\rangle\sum_{l=0}^{k-1}{k\choose l}\lambda^l\langle\xi,Z\rangle^{k-l-1}Z\otimes\xi\\
		&-Z(\lambda)\sum_{l=1}^{k-1}{k\choose l}\lambda^l\langle\xi,Z\rangle^{k-l-1}\langle\xi,X\rangle Z\otimes\xi\Big],
	\end{split}
	\]
	where, in the last sum, $\langle\xi,X\rangle Z\otimes\xi$ comes from the  computation
	\[
	\left((Z\otimes\xi)\circ (X\otimes\xi)\right)(Y)=\langle\xi,Y\rangle(Z\otimes\xi)(X)=\langle\xi,Y\rangle\langle\xi,X\rangle Z
	\qquad\forall Y\in\mathfrak X(M),
	\]
	which yields the identity $(Z\otimes\xi)\circ(X\otimes\xi)=\langle\xi,X\rangle Z\otimes\xi$. Now, recalling that $\text{Tr}(Z\otimes\xi)=\langle\xi,Z\rangle$ and $\text{Tr}(I)=3$, the last term of the previous chain of identities becomes
	\[
	\begin{split}
		&\frac12\left(3Z(\lambda)\lambda^k\langle\xi,X\rangle-\lambda^kZ(\lambda)\langle\xi,X\rangle+\cancel{Z(\lambda)\langle\xi,X\rangle\sum_{l=0}^{k-1}{k\choose l}\lambda^l\langle\xi,Z\rangle^{k-l}}\cancel{-Z(\lambda)\langle\xi,X\rangle\sum_{l=0}^{k-1}{k\choose l}\lambda^l\langle\xi,Z\rangle^{k-l}}
		\right)\\
		&=Z(\lambda)\lambda^k\langle\xi,X\rangle,
	\end{split}
	\]
	which entails $\phi_k=Z(\lambda)\lambda^k\xi$.
\end{proof}
With this result at hand, we can go back to the
\begin{proof}(of 
	Proposition \ref{thm:inv3}) We want to prove that $\{I_k,I_j\}=0$ for all $k,j$ with $k>j$. 
	Since the 1-forms $\phi_k$ are multiple of $\xi$, see \eqref{eq:phik}, and $\pi^\sharp\xi=0$, see (\ref{prop:xi2}),
	the right-hand side of \eqref{eq:recursionphi} is 
	zero. In other words, that identity collapses to
	\[
	\{I_k,I_j\}=\{I_{k-1},I_{j+1}\}.
	\] 
	If $k-1=j+1$ we are done, otherwise one can apply again this argument, which, after a finite number of steps, will entail the thesis.
\end{proof}

\begin{remark} 
	\label{rem:torsion-of-k-power}
	Using (\ref{eq: T_N computation}) and (\ref{eq:Nalk}), one obtains the following 
	generalization of (\ref{eq: T_N 3D PqN}), giving the Nijenhuis torsion of the $k$-th power of the tensor 
	$N=\lambda I+Z\otimes\xi$:
	\begin{equation}
		T_{N^k}(X,Y)=f_k Z(\lambda^k)i_\xi(X\wedge Y).\label{eq: T_N 3D PqN-k}
	\end{equation}
	This could also be deduced from the general expression in Remark 3 of \cite{BogoTens}.
\end{remark}

We close this subsection with the following remark, where we compare the result just obtained with Theorem \ref{theo:inv}, giving sufficient conditions for the involutivity of a (general) PqN manifold.
\begin{remark} \label{rem:FMP2024}
	Applying $\pi^\sharp$ to both sides of \eqref{eq:genrec} and using again (\ref{prop:xi2}) and Lemma \ref{lem:phik}, one arrives to the identity $X_{k+1}=N X_k$, which holds for all $k\geq 1$ and which, if iterated, gives $X_{k+1}=N^{k}X_1$ for all $k\geq 1$.
	This observation implies that the vector fields $Y_k$ appearing in item (b) of Theorem \ref{theo:inv} are identically zero. 
	On the other hand, one can find a 2-form $\Omega$ such that $\phi=-Z(\lambda)\mathcal V=-2dI_1\wedge\Omega$, see for example \cite{Pham}, Chapter 3, Section 1.1. From these observations we can deduce that both conditions of Theorem \ref{theo:inv} are satisfied, and we obtain another proof of 
	Proposition \ref{thm:inv3}.  
\end{remark}

\subsection{Recollection on the Haantjes torsion}\label{ss:haantjes}

To put the definition \eqref{eq:haant} of the Haantjes torsion in context, we need to introduce a few notions, which we enclose in the following 

\begin{defi}[See \cite{tempestatondo} and \cite{BenentiChanuRastelli}]
\begin{itemize}
\item[]
\item A frame over the open set $U\subset M$ is a collection ${\mathcal X}:=\{X_1,\dots,X_n\}$ of $n=\dim M$ vector fields which generate the tangent space at each point of $U$. 
\item Two frames $\mathcal X:=\{X_1,\dots,X_n\}$ and $\mathcal Y:=\{Y_1,\dots,Y_n\}$ over $U$ are called equivalent if there exists a collection of never vanishing smooth functions $f_1,\dots,f_n$ over $U$ such that $X_i=f_i Y_i$, for all $i=1,\dots,n$.
\item A 
frame $\mathcal X=\{X_1,\dots,X_n\}$ over $U$ is called integrable if for all $p\in U$ there exists a chart around $p$, 
say $(V,x_1,\dots,x_n)$, such that the restriction of $\mathcal X$ to $V$ is equivalent to the frame $\{\frac{\partial}{\partial x_1},\dots,\frac{\partial}{\partial x_n}\}$.
\item A tensor field $N$ of type $(1,1)$ is called semi-simple if each point $p\in M$ has a neighborhood $U$ where one can find a 
frame 
formed by eigenvectors of $N$. Such a frame is said to be an eigenframe for $N$.
\end{itemize}
\end{defi}

It is also worth recalling 
\begin{prop}
A reference frame $\mathcal X=\{X_1,\dots,X_n\}$ on $U$ is integrable if and only if one of the two following equivalent conditions is verified:
\begin{enumerate}
\item For every $i,j\in\{1,\dots,n\}$, the distribution generated by $\{X_i,X_j\}$ is Frobenius integrable.
\item The distribution generated by $\mathcal X\backslash\{X_i\}$ is Frobenius integrable.
\end{enumerate}
\end{prop}

The integrability of eigenframes for a given $(1,1)$ tensor field is controlled by the Haantjes torsion of $N$. 
More precisely, the vanishing of the Haantjes torsion of $N$ is a necessary and sufficient condition for the integrability of its eigenframes, see Theorem 25 in \cite{tempestatondo} and references therein. 

\section{Final comments and outlook}\label{sec:comfin} PN and bi-Hamiltonian structures has been proven crucial in the analysis of the complete integrability of classical Hamiltonian systems. In the series of papers \cite{FMOP2020,FMP2023,FMP2024}, it has been observed that under some stringent assumptions, one could use the notion of a PqN structure to describe the complete integrability of some classical and important integrable system from a geometrical point of view. To this end, the notion of involutive PqN structure was introduced and the one of a deformation of a PqN structure was analyzed.

In this paper first we characterize the PqN structures defined on oriented three dimensional manifolds whose underlying Poisson tensor never vanishes, see Theorem \ref{theo:charN}. Then we show that every such PqN structure is involutive, see Proposition \ref{thm:inv3} and that, \emph{locally}, it comes from a deformation of a PN one, see Proposition \ref{prop:local}. Theorem \ref{theo:haan} links this class of PqN structures with the Haantjes ones, while Theorem \ref{theo:Len-Ma-Ge} is a first attempt to build a bridge between the PqN structures and the so called generalized Lenard-Magri chains.

We believe that the results obtained in this work, although interesting in themselves, are preliminary to future research.
A few directions we would like to explore are the following ones.

First, we would like to investigate the possibility to extend to higher dimension the results enclosed in Section \ref{sec:3}. More precisely, as a preliminary checks, one could try to see if $\mathbb R^{2n+1}$, $n>1$, carries a PqN structure whose (1,1) tensor has the analogue form of \eqref{eq:thanksRef} and whose Poisson tensor, in Darboux coordinates $(x^1,\dots,x^n,y_1,\dots,y_n,z)$, is of the type $\pi=\sum_{i=1}^n\frac{\partial}{\partial x^i}\wedge\frac{\partial}{\partial y_i}$. If such structures exist, one could ask if for them the results in Propositions \ref{prop:local} and \ref{thm:inv3} hold true. 

Another direction we would like to pursue is to investigate under what assumptions a PqN structure in dimension greater than 3 is \emph{Haantjes torsion free}, see Theorem \ref{theo:haan}. We think that anything in this direction could be relevant to connect the theory of PqN structures to the one of dynamical systems of hydrodynamical type, see for example \cite{DN1,DN2}. It is already an interesting problem to look for systems of hydrodynamical type which can be defined from the PqN structures described in this paper, see \cite{FMS}. 

Finally, it could be interesting to see if the result  in Subsection \ref{ss:tempton}
could be extended to higher dimension.




\end{document}